\documentclass[11 pt]{amsart}
\usepackage[english]{babel}
\usepackage{amsthm}
\usepackage{graphicx}
\usepackage{amsfonts}
\usepackage{eufrak}
\usepackage{amsopn}
\usepackage{amsmath, amsfonts, amssymb, amscd, amsthm}
\usepackage{yhmath}
\usepackage{setspace}
\usepackage{mathtools}
\usepackage[all]{xy}
\newtheorem{theorem}{Theorem}[section]
\newtheorem{lemma}[theorem]{Lemma}
\newtheorem{proposition}[theorem]{Proposition}

\theoremstyle{definition}

\theoremstyle{remark}

\theoremstyle{example}

\theoremstyle{note}

\numberwithin{equation}{section}

\DeclareMathOperator{\Hom}{Hom}

\DeclareMathOperator{\Ind}{Ind}
\DeclareMathOperator{\ch}{char}
\DeclareMathOperator{\GL}{GL}
\DeclareMathOperator{\SL}{SL}
\DeclareMathOperator{\Ker}{Ker}
\DeclareMathOperator{\Gal}{Gal}
\DeclareMathOperator{\Image}{Im}
\DeclareMathOperator{\Span}{Span}

\DeclareMathOperator{\Sp}{Sp}
\DeclareMathOperator{\SO}{SO}
\DeclareMathOperator{\Id}{I}

\DeclareMathOperator{\Ortho}{O}

\begin{document}
\title{Self-dual representations with vectors fixed under an Iwahori subgroup}
%\title[Free and quasi-convex]{Quasi-Convex Free Polynomials}
\author{Kumar Balasubramanian}
%\thanks{${}^1$ Research supported by NSF grants DMS 0758306 and 1101137}
\address{Department of Mathematics,
Indian Institute of Science Education and Research Bhopal,
Bhopal, Madhya Pradesh, India.}
\email{bkumar@iiserb.ac.in}

\maketitle

\section*{Abstract}

Let $G$ be the group of $F$-points of a split connected
reductive $F$-group over a non-Archimedean local
field $F$ of characteristic 0. Let $\pi$ be an irreducible smooth self-dual representation of $G$. The
space $W$ of $\pi$ carries a non-degenerate $G$-invariant bilinear
form $(\,,\,)$ which is unique up to scaling. The form is easily seen to be symmetric or skew-symmetric and we set $\varepsilon({\pi})=\pm 1$ accordingly.
In this article, we show that $\varepsilon{(\pi)}=1$ when $\pi$ is a generic representation of $G$ with non-zero vectors fixed under an Iwahori subgroup $I$.\\

\section*{Introduction}

Let $G$ be a group and $(\pi,V)$ be an irreducible representation of $G$ such that $\pi \simeq \pi^{\vee}$ ($\pi^{\vee}$ denotes the dual or contragredient of $\pi$). This isomorphism gives rise to a non-degenerate $G$-invariant bilinear form on $V$ which is unique up to scalars, and consequently is either symmetric or skew-symmetric. Accordingly, we set
\begin{equation*}
\varepsilon{(\pi)} =
\begin{cases}
\;\;\;1 & \text{if the form is symmetric}\\
-1 & \text{if the form is skew-symmetric}
\end{cases}
\end{equation*}
and call it the sign of $\pi$. In this article, we study this sign for a certain class of representations of a reductive $p$-adic group $G$. To be more precise, we study the sign for representations with non-zero vectors fixed under an Iwahori subgroup in $G$. The structure of representations with Iwahori fixed vectors is well understood and this is one of the principal reasons we restrict our analysis to this particular class of representations. \\ %and exhibits many of the complications that occur when studying representations with fixed vectors under other compact open subgroups of $G$.

\section*{Overview of the problem}

Let $G$ be a group and $(\pi,V)$ be an irreducible self-dual complex representation of $G$. We say that the representation $\pi$ is realizable over the real numbers if there exists a $G$-invariant real subspace $W$ of $V$ such that $V\simeq \mathbb{C}\otimes_{\mathbb{R}} W$ as representations of $G$. A classical problem in representation theory was to determine when such a $W$ exists. For a finite group $G$, this problem was settled by Frobenius and Schur more than a century ago. They showed that $\pi$ is realizable over the real numbers precisely when $\varepsilon{(\pi)}=1$. They also gave a formula to compute $\varepsilon{(\pi)}$ in terms of the character $\chi_{\pi}$ of the representation $\pi$. They showed that
\begin{displaymath}
\varepsilon{(\pi)}= \frac{1}{|G|}\sum_{g\in G}\chi_{\pi}(g^{2}).
\end{displaymath}

The sign $\varepsilon{(\pi)}$ has been fairly extensively studied for connected compact Lie groups and finite groups of Lie type. In this setting, the sign is sometimes referred to as the Frobenius-Schur indicator of $\pi$. There is a lot of literature available on computing these signs for such groups. For connected compact Lie groups the sign can be computed using the dominant weight attached to the representation $\pi$ (see \cite{BroTom} pg. 261-264). In \cite{Gow[2]}, Gow showed that for $q$ a power of an odd prime and $F_{q}$ the finite field with $q$ elements, irreducible self-dual complex representations of $\SO(n,F_{q})$ are always realizable over $\mathbb{R}$. He also showed that the same is true for any non-faithful representation of $\Sp(n,F_{q})$. The proofs involve a detailed analysis of the conjugacy classes of these groups and are computationally quite complicated. In \cite{Pra[1]}, Prasad introduced an elegant idea to compute the sign for a certain class of representations of finite groups of Lie type. These representations are called generic. He used this idea to determine the sign for many finite groups of Lie type, avoiding tedious conjugacy class computations. In a subsequent paper \cite{Pra[2]}, he extended this idea to representations of a reductive $p$-adic group $G$ and computed the sign for generic representations of certain classical groups in some cases. In \cite{Vin[1]}, Vinroot used Prasad's idea along with other techniques to compute the sign for an irreducible self-dual representation of $\GL(n,F)$ where $F$ is a $p$-adic field.\\

In this article, we determine the sign $\varepsilon{(\pi)}$ when $\pi$ is a certain type of representation of an arbitrary connected reductive $p$-adic group $G$ which is split over the underlying $p$-adic field $F$. Suppose $K$ is the maximal compact open subgroup of $G$ and $\pi$ has non-zero vectors fixed under $K$. In this situation, it is easy to see that $\varepsilon{(\pi)}$ is always 1. A natural question is what is $\varepsilon{(\pi)}$ if $\pi$ has non-zero Iwahori fixed vectors. There is enough evidence that the sign is one in the Iwahori fixed case. We have not been able to prove the result in complete generality. However, we do address a particular case of this problem. To be more precise, we prove the following theorem.\\

\begin{theorem}[Main Theorem] Let $(\pi,W)$ be an irreducible smooth self-dual representation of $G$ with non-zero vectors fixed under an Iwahori subgroup in $G$. Suppose that $\pi$ is also generic. Then $\varepsilon{(\pi)}=1$.\\
\end{theorem}

\section{Self-dual representations and signs}\label{sign of a rep}

In this section, we introduce and briefly discuss the notion of signs associated to self-dual representations.

\subsection{Sign of $\pi$} Let $F$ be a non-Archimedean local field and $G$ be the group of $F$-points of a connected reductive algebraic group. Let $(\pi,W)$ be a smooth irreducible representation of $G$. We write $(\pi^{\vee}, W^{\vee})$ for the smooth dual or contragredient of $(\pi,W)$ and $\langle \, , \, \rangle$ for the canonical non-degenerate $G$-invariant pairing on $W \times W^{\vee}$ (given by evaluation). Let $s: (\pi,W)\rightarrow (\pi^{\vee}, W^{\vee})$ be an isomorphism. The map $s$ can be used to define a bilinear form on $W$ as follows
\begin{displaymath}
(w_{1}, w_{2})= \langle w_{1}, s(w_{2})\rangle, \quad \forall w_{1}, w_{2}\in W.
\end{displaymath}
It is easy to see that $(\,,\,)$ is a non-degenerate $G$-invariant form on $W$, i.e., it satisfies,
\begin{equation*}
(\pi(g)w_{1}, \pi(g)w_{2})= (w_{1}, w_{2}), \quad \forall w_{1}, w_{2}\in W.
\end{equation*}

Let $(\, , \,)_{\ast}$ be a new bilinear form on $W$ defined by
\begin{equation*}
(w_{1} , w_{2})_{\ast} = (w_{2}, w_{1})
\end{equation*}

Clearly, this form is again non-degenerate and $G$-invariant. It follows from Schur's Lemma that
\begin{equation*}\label{dual}
(w_{1} , w_{2})_{\ast} = c(w_{1}, w_{2})
\end{equation*}
for some non-zero scalar $c$. A simple computation shows that $c\in \{\pm 1\}$. Indeed,
\begin{displaymath}
(w_{1},w_{2}) = (w_{2},w_{1})_{\ast} = c(w_{2},w_{1}) = c(w_{1},w_{2})_{\ast} = c^{2}(w_{1},w_{2}).
\end{displaymath}
We set $c=\varepsilon{(\pi)}$. It clearly depends only on the equivalence class of $\pi$. In sum, the form $(\,,\,)$ is symmetric or skew-symmetric and the sign $\varepsilon{(\pi)}$ records its type.

\section{Representations of some classical groups}\label{reps of classical groups}
We use a theorem of Waldspurger to show that many representations of classical groups are self-dual. Throughout this section, we let $F$ be a non-Archimedean local field of characteristic $\neq 2$ and $W$ be a finite dimensional vector space over $F$. We write $\mathfrak{O}$ for the ring of integers in $F$, $\mathfrak{p}$ for the unique maximal ideal of $\mathfrak{O}$ and $k$ for the residue field. We let $\langle\, , \, \rangle$ to be a non-degenerate symmetric or skew-symmetric form on $W$. We take \begin{displaymath} G=\{g\in \GL(W) \mid \langle gw,gw' \rangle=\langle w, w' \rangle\} \end{displaymath} For $x\in \GL(W)$ such that $xGx^{-1}=G$ and $(\pi,V)$ a representation of $G$, we let $\pi^{x}$ denote the representation of $G$ defined by conjugation (i.e., $\pi^{x}(g)=\pi(xgx^{-1})$.\\

We recall the statement of Waldspurger's theorem and refer the reader to Chapter 4.II.1 in \cite{MoeWalVig} for a proof.

\begin{theorem}[Waldspurger]\label{waldspurger} Let $\pi$ be an irreducible admissible representation of $G$ and $\pi^{\vee}$ be the smooth-dual or contragredient of $\pi$. Let $x\in \GL(W)$ be such that $\displaystyle \langle xw,xw'\rangle = \langle w',w\rangle,\, \forall\, w,w'\in W$. Then $\pi^{x}\simeq \pi^{\vee}$.
\end{theorem}
\begin{proof} See chapter 4.II.1 in \cite{MoeWalVig}.
\end{proof}

\subsection{Orthogonal groups} Suppose the form $\langle\,,\,\rangle$ is symmetric so that $G$ is the orthgonal group $\Ortho(W)$. Let $\pi$ be any irreducible admissible representation of $G$. Then $x=1\in G$ satisfies $\langle xw,xw'\rangle = \langle w',w\rangle, \forall \,w,w'\in W$. Now using Waldspurger's Theorem, it follows that $\pi\simeq \pi^{\vee}$. So in the case of orthogonal groups every irreducible representation $\pi$ is self-dual.
\subsection{Special orthogonal groups} Suppose the dimension of $W$ is odd. Take $G=\SO(W)=\Ortho(W)\,\cap\, \SL(W)$ and $\pi$ to be an irreducible admissible representation of $G$. Since $\Ortho(W)\simeq \SO(W)\times \{\pm 1_{W}\}$, it follows that there exists an irreducible representation $\tilde{\pi}$ of $\Ortho(W)$ such that $\tilde{\pi}\simeq \pi \otimes \chi$ (where $\chi$ is a character of $\{\pm 1_{W}\}$). Since $\chi=\chi^{-1}$ and $\tilde{\pi} \simeq \tilde{\pi}^{\vee}$ it follows that $\pi\simeq \pi^{\vee}$.
\subsection{Symplectic groups} Suppose that $G$ is the symplectic group $Sp(W)$ or $Sp(n,F)$. We show that a certain class of representations of $G$ is always self-dual. To be more precise, we prove,

\begin{theorem}\label{symplectic case} Let $(\pi,V)$ be an irreducible admissible representation of $G$ with non-zero vectors fixed under an Iwahori subgroup $I$ in $G$. Then $\pi\simeq \pi^{\vee}$.
\end{theorem}

Consider $x= \bigl[\begin{smallmatrix}-\Id & 0 \\\,\,\, 0 & \Id \end{smallmatrix}\bigr]\in \GL(W)$ (where $\Id$ is the $n\times n$ identity matrix). It is easy to see that $\langle xw, xw' \rangle = \langle w', w \rangle$. By Theorem~\ref{waldspurger}, $\pi^{x}\simeq \pi^{\vee}$. To prove $\pi\simeq \pi^{\vee}$, it suffices to show that $\pi \simeq \pi^{x}$. Observe that $xIx^{-1}=I$ and $\pi^{I}=(\pi^{x})^{I}$. Since $\pi^{I}=(\pi^{x})^{I}\neq 0$, they can be realized as simple modules over $\mathcal{H}(G,I)$ (where $\mathcal{H}(G,I)$ is the Iwahori-Hecke algebra). Let $f\bullet v$ and $f\star v$ denote the action of $\mathcal{H}(G,I)$ on $\pi^{I}$ and $(\pi^{x})^{I}$ respectively. It will follow that $\pi \simeq \pi^{x}$ if $\pi^{I}$ and $(\pi^{x})^{I}$ are equivalent as $\mathcal{H}(G,I)$-modules. We establish this equivalence, by showing that the map $\phi=1_{V}$ (identity map on $V$) defines an intertwining map between $\pi^{I}$ and $(\pi^{x})^{I}$. \\

Before we continue, we fix a collection of coset representatives for the affine Weyl group $\widetilde{\mathcal{W}}$ and record two lemmas we need.\\

Let $B_{i,i+1}, i=1,2,\ldots, n-1$, be the $n\times n$ matrix
\begin{align*}
B_{i,i+1}& = \begin{bmatrix}
1 \\
  & \ddots \\
  &       & 0 & 1 \\
  &       & 1 & 0 & \\
  &       &   &   & \ddots &\\
  &       &   &   &        & 1\\
\end{bmatrix}
\left\}
\begin{array}{c}
i \\
i+1%
\end{array}%
\right.
\end{align*}
\vspace{0.5 cm}\\
and $w_{n}$ and $w_{\ell},\, \ell\in \mathbb{Z}^{n}$, be the following $2n \times 2n$ matrices
\vspace{0.5 cm}
\begin{align*}
w_{n} &= \begin{bmatrix}
 1 & & \\
 & \ddots & \\
 & & 0 & 1\\
 & & -1  & 0\\
 & &  &  & \ddots \\
 & &  &  &  & 1\\
\end{bmatrix}
\left\}
\begin{array}{c}
n \\
n+1%
\end{array}%
\right., &\\\\\\
w_{\ell} &= \begin{bmatrix}
\varpi^{l_{1}} & & \\
 & \ddots & \\
 & & \varpi^{l_{n}}\\
 & &  & \varpi^{-l_{n}}\\
 & &  &  & \ddots \\
 & &  &  &  & \varpi^{-l_{1}}\\
\end{bmatrix}, l_{i}\in \mathbb{Z}.
\end{align*}
\vspace{0.5 cm}

%\begin{equation*}
%E =\left[
%\begin{array}{cccccc}
%1          &        &       &        &    \\
%           &\ddots  &       &        &   &   \\
% &         &\,\,\,0 & 1     &        &    \\
% &         & -1     & 0     &        &  \\
% &         &        &       & \ddots &    \\
%
%
% &         &        &       &        & 1
%\end{array}%
%\right] \left\}
%\begin{array}{c}
%n \\
%n+1%
%\end{array}%
%\right.
%\end{equation*}
%\newline
The group $
\bigg \langle w_{\ell}, w_{n},
w_{i}=\biggl[\begin{smallmatrix}
B_{i,i+1} & 0\\
0         & B_{i,i+1}\\
\end{smallmatrix}\biggr]
\,\, \bigg |\,\, \ell\in \mathbb{Z}^{n}, i=1,2,\ldots, n-1 \bigg \rangle
$ contains a collection of coset representatives for the affine Weyl group $\widetilde{\mathcal{W}}$.

\begin{lemma}\label{action relation} For $f\in \mathcal{H}(G,I)$, let $f^{x}\in \mathcal{H}(G,I)$ be the function $f^{x}(g)=f(x^{-1}gx)$.
The following statements are true.
\begin{enumerate}
\item [(i)] $f \star v = f^{x} \bullet v$.
\item [(ii)] For $g=i_{1}wi_{2}\in G$ (Bruhat Decomposition), $f^{x}(g)=f(xwx^{-1})$ and $f(g)=f(w)$.
\end{enumerate}
\end{lemma}
\begin{proof} Clearly $f\star v = f^{x}\bullet v$. Indeed,
\begin{align*}
f \star v &= \int_{G}f(g)\pi^{x}(g)vdg\\
&= \int_{G}f(g)\pi(xgx^{-1})vdg\\
&= \int_{G}f(x^{-1}gx)\pi(g)vdg\\
&= f^{x} \bullet v.
\end{align*}
For (ii),
\begin{align*}
f^{x}(g) &= f(xi_{1}wi_{2}x^{-1})\\
&= f(xi_{1}x^{-1}xwx^{-1}xi_{2}x^{-1})\\
&= f(ixwx^{-1}i')\\
&= f(xwx^{-1}),
\end{align*}
and
\begin{align*}
f(g)&= f(i_{1}wi_{2})\\
&= f(w).
\end{align*}
\end{proof}
%\vspace{0.5 cm}

\begin{lemma}\label{action of x} For $w\in \widetilde{W}$, $xwx^{-1}\in IwI$.
\end{lemma}

\begin{proof} If $w\in \widetilde{W}$, we can write $w=u_{1}u_{2}\ldots u_{n}$ for $u_{k}\in \{w_{\circ}, w_{i}, w_{\ell}\}$. In this case we say $w$ has length $n$ and denote it as $l(w)=n$. We will use induction on the length $l(w)$ to show that $xwx^{-1}\in IwI$. Suppose that $l(w)=1$. A simple computation shows that conjugation by the element $x$ fixes the elements $w_{i}$, $w_{\ell}$ and fixes $w_{n}$ up to multiplication by elements in $T_{\circ}\subset I$, i.e., $xw_{i}x^{-1}=w_{i}$ for $i=1,\ldots, n-1$, $xw_{\ell}x^{-1}=w_{\ell}$ for $\ell \in \mathbb{Z}^{n}$ and $xw_{n}x^{-1}=tw_{n}t^{-1}$ for some $t\in T_{\circ}\subset I$. Suppose $l(w)=2$, i.e., $w=u_{1}u_{2}$ for $u_{1},u_{2}\in \{w_{0},w_{i}, w_{\ell}\}$ such that $xu_{1}x^{-1}=t_{1}u_{1}t_{1}^{-1}$ and $xu_{2}x^{-1}=t_{2}u_{2}t_{2}^{-1}$, $t_{1},t_{2}\in T_{\circ}$. In this case, we have
\begin{align*}
xwx^{-1} &= xu_{1}x^{-1}xu_{2}x^{-1}\\
&= t_{1}u_{1}t_{1}^{-1}t_{2}u_{2}t_{2}^{-1}\\
&= t_{1}u_{1}u_{2}\underbrace{u_{2}^{-1}t_{1}^{-1}u_{2}}_{\in T_{\circ}}\underbrace{u_{2}^{-1}t_{2}u_{2}}_{\in T_{\circ}}t_{2}^{-1}\\
&= t_{1}u_{1}u_{2}t_{1}'
\end{align*}
where $t,t'\in T_{\circ}$.\\

Assume that the result is true for all words $w$ such that $l(w)\leq n-1$. Suppose $w=u_{1}u_{2}\dots u_{n}$. Now
\begin{align*}
xwx^{-1} &= xu_{1}u_{2}\dots u_{n-1}x^{-1}xu_{n}x^{-1}\\
&= tu_{1}u_{2}\dots u_{n-1}t^{-1}xu_{n}x^{-1}\\
&= twt' \quad (\text{by previous case})
\end{align*}
where $t,t'\in T_{\circ}$.
\end{proof}

We now prove Theorem~\ref{symplectic case}. To prove $\phi=1_{V}$ is an intertwining map, we need to show that $f \bullet v = f \star v = f^{x} \bullet v$. By Lemma~\ref{action relation}, it suffices to show that $f(xwx^{-1})=f(w)$, for all $w\in \widetilde{\mathcal{W}}$. By Lemma~\ref{action of x}, it follows that
conjugation by the element $x$ fixes every element in $\widetilde{W}$ (up to multiplication by elements in $I$). The result now follows.\\

\section{Results used in proof of main theorem}\label{main results}

In this section, we recall the important results used in the proof of the main theorem.

\subsection{Restriction of representations to subgroups}
We recall some results about restricting an irreducible representation to a subgroup. These results hold when $G$ is a locally compact totally disconnected group and $H$ is an open normal subgroup of $G$ such that $G/H$ is finite abelian. For a more detailed account, we refer the reader to \cite{GelKna} (Lemma 2.1, 2.3).

\begin{theorem}[Gelbart-Knapp]\label{Gelbart-Knapp-1} Let $\pi$ be an irreducible admissible representation of $G$. Suppose that $G/H$ is finite abelian. Then
\begin{enumerate}
\item[(i)] $\pi|_{H}$ is a finite direct sum of irreducible admissible representations of $H$.
\item[(ii)] When the irreducible constituents of $\pi|_{H}$ are grouped according to their equivalence classes as
\end{enumerate}
\begin{equation*}
\pi|_{H}=\bigoplus_{i=1}^{M} m_{i}\pi_{i}
\end{equation*}
with the $\pi_{i}$ irreducible and inequivalent, the integers $m_{i}$ are equal.
\end{theorem}

\begin{theorem}[Gelbart-Knapp]\label{Gelbart-Knapp-2} Let $G$ be a locally compact totally disconnected group and $H$ be an open normal subgroup of $G$ such that $G/H$ is finite abelian, and let $\pi$ be an irreducible admissible representation of $H$. Then
\begin{enumerate}
\item[(i)] There exists an irreducible admissible representation $\tilde{\pi}$ of $G$ such that $\tilde{\pi}|_{H}$ contains $\pi$ as a constituent.
\item[(ii)] Suppose $\tilde{\pi}$ and $\tilde{\pi}'$ are irreducible admissible representations of $G$ whose restrictions to $H$ are multiplicity free and contain $\pi$. Then $\tilde{\pi}|_{H}$ and $\tilde{\pi}'|_{H}$ are equivalent and $\tilde{\pi}$ is equivalent with $\tilde{\pi}'\otimes \chi$ for some character $\chi$ of $G$ that is trivial on $H$.
\end{enumerate}
\end{theorem}

\subsection{Unramified principal series and representations with Iwahori fixed vectors}
We define the notion of an unramified principal series representation and state an important characterization of representations with non-zero vectors fixed under an Iwahori subgroup due to Borel and Casselman. We refer to (\cite{Bor}, \cite{Cas}) for a proof.\\

Throughout this section, we let $G$ be the group of $F$-points of a connected reductive algebraic group defined and split over $F$. We write $T$ for a maximal $F$-split torus in $G$. We also fix a Borel subgroup $B$ defined over $F$ such that $B\supset T$ and write $U$ for the unipotent radical of $B$.\\

Let $(\rho,W)$ be a smooth representation of $T$. We can view $\rho$ as a smooth representation of $B$ which is trivial on $U$, and form the smooth induced representation $\Ind_{B}^{G}\rho$. The representations $\Ind_{B}^{G}(\mu)$, where $\mu$ is an unramified character of $T$ (i.e., $\mu|_{T_{\circ}}=1$ where $T_{\circ}$ is the $\mathfrak{O}$-points of $T$) are called the unramified principal series representations.

\begin{theorem}[Borel-Casselman]\label{Borel-Casselman} Let $(\pi,W)$ be any irreducible admissible representation of $G$. Then the following assertions are equivalent.
\begin{enumerate}
\item[(i)] There are non-zero vectors in $W$ invariant under $I$.
\item[(ii)] There exists some unramified character $\mu$ of $T$ such that $\pi$ imbeds as a sub-representation of $\Ind_{B}^{G}\mu$.
\end{enumerate}
\end{theorem}

\subsection{Prasad's idea for computing the sign} In \cite{Pra[2]}, Prasad gives a criterion to compute the sign for an irreducible self-dual generic representation of a $p$-adic group $G$. He shows that for generic representations the sign is determined by the value of the central character $\omega_{\pi}$ at a special central element. We recall his result below.

\begin{theorem}[Prasad]\label{Prasad's theorem} Let $K$ be a compact open subgroup of $G$. Let $s$ be an element of $G$ which normalizes $K$ and whose square belongs to the center of $G$. Let $\psi_{K}: K \rightarrow \mathbb{C}^{\times}$ be a one dimensional representation of $K$ which is taken to its inverse by inner conjugation action of $s$ on $K$. Let $\pi$ be an irreducible representation of $G$ in which the character $\psi_{K}$ of $K$ appears with multiplicity $1$. Then if $\pi$ is self-dual, $\varepsilon{(\pi)}$ is $1$ if and only if the element $s^{2}$ belonging to the center of $G$ operates by $1$ on $\pi$.
\end{theorem}

\subsection{Compact approximation of Whittaker models}

In this section, we recall an important result of Rodier which we need in the proof of the main theorem. We refer the reader to ~\cite{Rod} for a more detailed account of Rodier's results.

\begin{theorem}[Rodier]\label{Rodier} Let $\pi$ be an irreducible admissible representation of $G$ and $\psi$ be a non-degenerate character of $U$. There then exists a compact open subgroup $K$ and a character $\psi_{K}$ of $K$ such that \begin{displaymath}\displaystyle dim_{\mathbb{C}}\Hom_{K}(\pi, \psi_{K})= dim_{\mathbb{C}}\Hom_{U}(\pi, \psi).\end{displaymath} Therefore, if $\pi$ is generic, $dim_{\mathbb{C}}\Hom_{K}(\pi, \psi_{K})=1$.
\end{theorem}

\section{Main Theorem}\label{main result proof}

In this section, we prove the main theorem. We recall the statement below.

\begin{theorem}[Main Theorem] Let $(\pi,W)$ be an irreducible smooth self-dual representation of $G$ with non-zero vectors fixed under an Iwahori subgroup $I$ in $G$. Suppose that $\pi$ is also generic. Then $\varepsilon{(\pi)}=1$.
\end{theorem}

We first prove the result when $G$ has connected center. In this case, we use a result of Rodier (Theorem~\ref{Rodier}) to get a compact open subgroup $K$ and a character $\psi_{K}$ of $K$ which appears with multiplicity one in $\pi|_{K}$. We show that there exists an element $s\in T$ satisfying the hypotheses of Prasad's Theorem (Theorem~\ref{Prasad's theorem}). Finally we use the fact that $\pi$ has non-zero Iwahori fixed vectors to show that $\varepsilon{(\pi)}=1$.\\

When the center of $G$ is not connected, we construct a split connected reductive $F$-group $\tilde{G}$ with a maximal $F$-split torus $\tilde{T}$. The group $\tilde{G}$ has a connected center $\tilde{Z}$ and contains $G$ as a subgroup. We show that there exists an irreducible representation $\tilde{\pi}$ of $\tilde{G}$ that contains the representation $\pi$ with multiplicity one on restriction to $G$ and has non-zero vectors fixed under an Iwahori subgroup in $\tilde{G}$. The representation $\tilde{\pi}$ is not necessarily self-dual but is self-dual up to a twist by a character $\chi$ of $\tilde{G}$ which is trivial on $G$. We can still attach a sign $\varepsilon(\tilde{\pi})$ to $\tilde{\pi}$. Finally, we show that $\varepsilon{(\tilde{\pi})}=\varepsilon({\pi})$ and $\varepsilon({\tilde{\pi}})=1$.\\

%We first prove the result when $G$ has a connected center. When the center of $G$ is not connected, we construct a split connected reductive $F$-group $\tilde{G}$ with a maximal $F$-split torus $\tilde{T}$. The group $\tilde{G}$ has a connected center $\tilde{Z}$ and contains $G$ as a subgroup. We show that there exists an irreducible representation $\tilde{\pi}$ of $\tilde{G}$ that contains the representation $\pi$ with multiplicity one on restriction to $G$ and has non-zero vectors fixed under an Iwahori subgroup in $\tilde{G}$. Finally, we show that $\varepsilon{(\tilde{\pi})}=\varepsilon({\pi})$ and $\varepsilon({\tilde{\pi}})=1$.

Throughout this section, we let $G$ be the group of $F$-points of a connected reductive algebraic group defined and split over $F$. We write $T$ for a maximal $F$-split torus in $G$. We also fix a Borel subgroup $B$ defined over $F$ such that $B\supset T$. We write $U$ for the unipotent radical of $B$ (respectively $\bar{U}$ for the $T$-opposite of $U$) and fix a non-degenerate character $\psi$ of $U$ such that $\Hom_{U}(\pi,\psi)\neq 0$ ($\psi$ exists since $\pi$ is generic). We let $X$ and $X^{\vee}$ be the character and cocharacter groups of $T$. We write $\Phi$ and $\Phi^{\vee}$ for the set of roots and coroots and $\Delta$ for the set of simple roots of $T$. Since $T$ is $F$-split, we have unique subgroups $T_{\circ}$ and $T_{1}$ of $T$ such that $T=T_{\circ}\times T_{1}$. To be more precise, the isomorphism $F^{\times}\simeq \mathfrak{O}^{\times} \times \mathbb{Z}$ given by $x\varpi^{n} \mapsto (x,n)$ induces the following isomorphism
\begin{displaymath}
T\simeq X^{\vee}\otimes F^{\times}\simeq X^{\vee}\otimes \mathfrak{O}^{\times}\oplus X^{\vee}\otimes \mathbb{Z}
\end{displaymath}
and we take $T_{\circ}$ and $T_{1}$ to be the subgroups of $T$ such that $T_{\circ}\simeq X^{\vee}\otimes \mathfrak{O}^{\times}$ ($\alpha^{\vee}\otimes y \rightarrow \alpha^{\vee}(y)$) and $T_{1}\simeq X^{\vee}\otimes \mathbb{Z}$ ( $\alpha^{\vee}\otimes n \rightarrow \alpha^{\vee}(\varpi^{n})$). We have a similar decomposition for $\tilde{T}$ (i.e., $\tilde{T}=\tilde{T}_{\circ}\times \tilde{T}_{1}$). In what follows, we let $Z_{\circ}= Z\cap T_{\circ}$ and $Z_{1}= Z\cap T_{1}$ (respectively $\tilde{Z}_{\circ}= \tilde{Z}\cap \tilde{T}_{\circ}$ and $\tilde{Z}_{1}= \tilde{Z}\cap \tilde{T}_{1}$). We let $\omega_{\circ}= \omega_{\pi}|_{Z_{\circ}}$ and $\omega_{1}= \omega_{\pi}|_{Z_{1}}$.

\subsection{Center of $G$ is connected}

\subsection*{} In this section, we show the existence of an element $s\in T$ satisfying the conditions of Prasad's theorem (Theorem~\ref{Prasad's theorem}) and use it to compute the sign $\varepsilon{(\pi)}$ when the center of $G$ is connected.

\begin{lemma} Let $s\in T$ be such that $\alpha(s)=-1$ for all simple roots $\alpha\in \Delta$. The following are true.
\begin{enumerate}
\item[(i)] For $u\in U$, we have $\psi(sus^{-1})=\psi^{-1}(u)$.
\item[(ii)] The element $s^{2}$ belongs to the center of $G$.
\item[(iii)] The element $s$ normalizes the compact open subgroups $K_{m}$ and inner conjugation by $s$ takes the characters $\psi_{m}$ to its inverse.
\end{enumerate}
\end{lemma}

\begin{proof} Since $U$ is generated by $U_{\alpha},\, \alpha\in \Phi$, it is enough to show that $\psi(sus^{-1})=\psi^{-1}(u)$ for $u\in U_{\alpha}$. For $u=x_{\alpha}(\lambda)\in U_{\alpha}$ we have,
\begin{align*}
\psi(sus^{-1}) &= \psi(sx_{\alpha}(\lambda)s^{-1})\\
&= \psi(x_{\alpha}(\alpha(s)\lambda))\\
&= \psi(x_{\alpha}(-\lambda))\\
&= \psi^{-1}(x_{\alpha}(\lambda))\\
&= \psi^{-1}(u).
\end{align*}

For (ii), since $\alpha(s)=-1$ it is clear that $\alpha(s^{2})=1$ for all simple roots $\alpha \in \Delta$. Since $\displaystyle Z(G)=\bigcap_{\alpha\in \Delta} \Ker(\alpha)$, the result follows.\\

For (iii), It is easy to see that $s$ normalizes $U_{m}$, $\bar{U}_{m}$. It follows that $s$ normalizes $G_{m}$ and hence $K_{m}$. Let $k=k^{-}k^{\circ}k^{+}\in K_{m}$ (since $K_{m}=K^{-}_{m}K^{\circ}_{m}K^{+}_{m}$). We have
\begin{align*}\psi_{m}(sks^{-1}) &=\psi_{m}(sk^{-}s^{-1}sk^{\circ}s^{-1}sk^{+}s^{-1})\\
&= \psi(sk^{+}s^{-1})\\
&=\psi^{-1}(k^{+})\\
&= \psi_{m}^{-1}(k).
\end{align*}
\end{proof}
%\vspace{0.5 cm}

\begin{theorem}\label{sign in connected case}
Let $(\pi,W)$ be an irreducible smooth self-dual generic representation of $G$ with non-zero vectors fixed under an Iwahori subgroup $I$ in $G$. Suppose there exists an element $s\in
T_{\circ}$ such that $\alpha(s)=-1$ for all simple roots $\alpha$. Then $\varepsilon{(\pi)}=1$.
\end{theorem}
\begin{proof}
By Theorem~\ref{Prasad's theorem}, it is enough to show that
$\omega_{\pi}(s^{2})=1$ ($\omega_{\pi}$ is the central
character). Let $v\neq 0 \in \pi^{I}$. We have
$v=\pi(s^{2})v=\omega_{\pi}(s^{2})v$. From this it follows that $\omega_{\pi}(s^{2})=1$.
\end{proof}
%\vspace{0.5 cm}

\begin{theorem} \label{existence of the elt s}There exists $s\in T_{\circ}$ such that $\alpha(s)=-1$ for all the simple roots $\alpha$.
\end{theorem}
\begin{proof}
We know that $X^{\vee}\otimes F^{\times}\simeq T$, via $y\otimes \lambda \mapsto y(\lambda)$. Since $F^{\times}\simeq \mathfrak{O}^{\times} \rtimes \mathbb{Z}$, we see that $T\simeq X^{\vee}\otimes \mathfrak{O}^{\times}\oplus X^{\vee}\otimes \mathbb{Z}$. Now define $\displaystyle f: T\longrightarrow \prod_{\alpha_{i}\in \Delta}F^{\times}$ by \begin{displaymath}
f(y\otimes\lambda)= (\lambda^{\langle \alpha_{1},y\rangle}, \ldots, \lambda^{\langle \alpha_{k},y\rangle}).
\end{displaymath}
We will show that there exists $y\in X^{\vee}$ such that $\langle \alpha_{i}, y\rangle$ is an odd integer for every simple root $\alpha_{i}, i=1,\ldots, k$. Since $Z$ is connected $X/\mathbb{Z}\Phi$ is torsion free. Since $\Delta$ spans $\Phi$ we see that $\mathbb{Z}\Phi= \mathbb{Z}\Delta$. Consider the exact sequence
\begin{equation}\label{eq:2}
0 \longrightarrow \mathbb{Z}\Delta \longrightarrow X \longrightarrow X/\mathbb{Z}\Delta \longrightarrow 0.
\end{equation}
Since \eqref{eq:2} is an exact sequence of finitely generated free abelian groups, it is split. i.e., $X= \mathbb{Z}\Delta \oplus L$, where $L\simeq X/\mathbb{Z}\Delta$. Let $g\in \Hom_{\mathbb{Z}}(\mathbb{Z}\Delta, \mathbb{Z})$. Clearly, $g$ extends to an element of $\Hom_{\mathbb{Z}}(X,\mathbb{Z})$ (say trivial on $L$). Since $\Hom_{\mathbb{Z}}(X,\mathbb{Z})\simeq X^{\vee}$, there exists $y\in X^{\vee}$ such that $g=\langle -, y\rangle$. We now choose $h\in \Hom_{\mathbb{Z}}(\mathbb{Z}\Delta, \mathbb{Z})$ such that $h(\alpha_{i})$ is odd for every $\alpha_{i}, i= 1, 2, \ldots, k$. Then $h(\alpha_{i})= \langle \alpha_{i}, y\rangle$ is an odd integer. Now consider the element $y \otimes -1 \in X^{\vee}\otimes \mathfrak{O}^{\times}$. Let $s=y(-1)$. Then $s\in T_{\circ}$ clearly acts by $-1$ on all the simple root subgroups $U_{\alpha}$ of $U$, i.e.,
\begin{align*}
sx_{\alpha_{i}}(\mu)s^{-1} &= x_{\alpha_{i}}(\alpha_{i}(s)\mu)\\
&= x_{\alpha_{i}}(\alpha_{i}(y(-1))\mu)\\
&= x_{\alpha_{i}}((-1)^{\langle \alpha_{i},y \rangle}\mu)\\
&= x_{\alpha_{i}}(-\mu).
\end{align*}
\end{proof}

%________________________________________________________________________________________________________________________________________________________________

\subsection{Center of $G$ is not connected}

\subsubsection{Construction of $(\tilde{G}, \tilde{T})$} Let $q:X\rightarrow X/\mathbb{Z}\Phi$ be the canonical quotient map.
Choose a free abelian group $L$ of finite rank such that there exists a surjective map $p:L\rightarrow X/\mathbb{Z}\Phi$. Let $p_{1}$ and $p_{2}$ be the projection maps from $X\times L$ onto $X$ and $L$ respectively. Let
\begin{displaymath}
\tilde{X}= \{(x,l)\in X\times L\mid \,q(x)= p(l)\}.
\end{displaymath}
Clearly $\tilde{X}$ is a free abelian group of finite rank. %(since it is a subgroup of the free abelian group $X\times L$ of finite rank).
Let $\tilde{\Phi}= \{(\alpha,0) \mid  \alpha \in \Phi\}$. The map $\alpha\mapsto (\alpha,0)$ induces an injection $\mathbb{Z}\Phi\hookrightarrow \tilde{X}$ and we identify its image under the map with $\mathbb{Z}\tilde{\Phi}$. Let $\tilde{X}^{\vee}= \Hom_{\mathbb{Z}}(\tilde{X}, \mathbb{Z})$. Given $\tilde{\alpha}\in \tilde{\Phi}$ we want to describe $\tilde{\alpha}^{\vee}\in \tilde{\Phi}^{\vee}\subset \tilde{X}^{\vee}$. Now $\tilde{\alpha}=(\alpha,0)$ for some $\alpha\in \Phi$. For this $\alpha$, there exists $\alpha^{\vee}\in \Phi^{\vee}\subset X^{\vee}\simeq \Hom_{\mathbb{Z}}(X,\mathbb{Z})$. Let $\tilde{x}=(x,0)\in \tilde{X}$. Define
 $\tilde{\alpha}^{\vee}(\tilde{x})= \tilde{\alpha}^{\vee}((x,0))= \alpha^{\vee}(p_{1}(\tilde{x}))$. Clearly $\tilde{\alpha}^{\vee}\in \Hom_{\mathbb{Z}}(\tilde{X}, \mathbb{Z})$. It is easy to see that $(\tilde{X}, \tilde{\Phi}, \tilde{X}^{\vee}, \tilde{\Phi}^{\vee})$ is a root datum. By the classification theorem for split groups, the existence of $(\tilde{G}, \tilde{T})$ follows. Since $\tilde{X}/\mathbb{Z}\tilde{\Phi}\hookrightarrow L$, it follows that it is torsion free and the center $\tilde{Z}$ of $\tilde{G}$ is connected.

\subsubsection{Extension of the central character}
In this section, we show that there exists a character $\nu$ of $\tilde{Z}$ which extends the central character $\omega_{\pi}$ and satisfies $\nu^{2}=1$.

\begin{lemma} There exists an unramified character $\mu: T \rightarrow \mathbb{C}^{\times}$ such that $\mu |_{Z}= \omega_{\pi}$.
\end{lemma}

\begin{proof}
Since $\pi$ has Iwahori fixed vectors, there exists an unramified character $\mu$ of $T$ such that $\pi\hookrightarrow \Ind_{B}^{G}\mu$. Let $(\rho, E)$ be an irreducible sub-representation of $\Ind_{B}^{G}\mu$ that is isomorphic to $\pi$. Let $x\in Z, f\in E, g\in G$. Clearly,
\begin{equation}\label{eqn1}
(\rho(x)f)(g)= f(gx)= f(xg)= \mu(x)f(g)
\end{equation}
On the other hand,
\begin{equation}\label{eqn2}
(\rho(x)f)(g)= \omega_{\rho}(x)f(g)
\end{equation}
From (~\ref{eqn1}) and (~\ref{eqn2}) it follows that $\omega_{\rho}(x)= \mu(x)= \omega_{\pi}(x)$.
\end{proof}
%\vspace{0.5 cm}

Since $\mu$ is unramified it follows that $\mu|_{Z}=\omega_{1}$ (since $\mu|_{Z_{\circ}}=1$). If we can extend $\omega_{1}$ to a self-dual character $\tilde{\omega}_{1}$ of $\tilde{Z}_{1}$ then we get a self-dual character $\nu$ of $\tilde{Z}$ extending the central character $\omega_{\pi}$. We record the result in a lemma below.

\begin{lemma}\label{defining nu} Suppose that $\tilde{\omega}_{1}$ is an extension of\, $\omega_{1}$ to $\tilde{Z}_{1}$. Then there exists $\nu: \tilde{Z}\rightarrow \{\pm 1\}$ such that $\nu$ extends $\omega_{\pi}$.
\end{lemma}

\begin{proof} For $\tilde{z}=\tilde{z}_{0}\tilde{z}_{1}\in \tilde{Z}$, $\tilde{z}_{0}\in \tilde{Z}_{\circ}, \tilde{z}_{1}\in \tilde{Z}_{1}$ define $\nu(\tilde{z})= \tilde{\omega}_{1}(\tilde{z}_{1})$. Clearly $\nu$ is a well-defined character of $\tilde{Z}$ and $\nu|_{Z}=\tilde{\omega}_{1}|_{Z}= \omega_{1}=\mu|_{Z}=\omega_{\pi}$.
\end{proof}
%\vspace{0.5 cm}

From Lemma~\ref{defining nu}, it follows that we can extend the central character $\omega_{\pi}$ to a self-dual character $\nu$ of $\tilde{Z}$ if there exists an extension $\tilde{\omega}_{1}$ of $\omega_{1}$. Consider the map $\omega_{1}': Z_{1}/Z_{1}^{2} \rightarrow \{\pm 1\}$ defined by $\omega_{1}'(aZ_{1}^{2})=\omega_{1}(a)$. Since $Z_{1}/Z_{1}^{2}$ is an elementary abelian $2$-group, $\omega_{1}'$ can be thought of as a $\mathbb{Z}_{2}$-linear map. If the natural map from $Z_{1}/Z_{1}^{2}$ to $\tilde{Z}_{1}/\tilde{Z}_{1}^{2}$ is an embedding, then we can extend the $\mathbb{Z}_{2}$-linear map $\omega_{1}'$ to a $\mathbb{Z}_{2}$-linear map $\tilde{\omega}_{1}'$ of $\tilde{Z}_{1}/\tilde{Z}_{1}^{2}$. Now defining $\tilde{\omega}_{1}(a)=\tilde{\omega}_{1}'(a\tilde{Z}_{2})$ gives us an extension of $\omega_{1}$. The natural map is an embedding precisely when $\tilde{Z}^{2}_{1}\cap Z_{1}\subset Z^{2}_{1}$. We record the result in the following lemma.

\begin{lemma} The natural map $Z_{1}/Z_{1}^{2}$ to $\tilde{Z}_{1}/\tilde{Z}_{1}^{2}$ is an embedding and $\omega_{1}$ extends to a character $\tilde{\omega}_{1}$ of $\tilde{Z}_{1}$.
\end{lemma}

\begin{proof} It is enough to show that $\tilde{Z}^{2}_{1}\cap Z_{1}\subset Z^{2}_{1}$. Consider $z_{1}\in \tilde{Z}^{2}_{1}\cap Z_{1}$. Clearly $z_{1}=\tilde{z}_{1}^{2}$ for some $\tilde{z}_{1}\in \tilde{Z}_{1}$. It is enough to show that $\tilde{z}_{1}\in T$ (since $\tilde{z}_{1}\in T$ implies $\tilde{z}_{1}\in \tilde{Z}_{1}\cap T= Z_{1}$ and $\tilde{z}^{2}_{1}\in Z_{1}$). Since $T\hookrightarrow \tilde{T}$ there exists a sub-torus $S$ such that $\tilde{T}=\tilde{T}_{\circ}\times \tilde{T}_{1}=T\times S$. Clearly, $\tilde{T}_{1}= T_{1}\times S_{1}$. Indeed,
\begin{align*}
\tilde{T}&= \tilde{T}_{\circ}\times \tilde{T}_{1}\\
&= T_{\circ}\times T_{1}\times S_{\circ}\times S_{1}\\
&= {T}_{\circ}\times {S}_{\circ}\times T_{1}\times S_{1}.\\
\end{align*}

Now $\tilde{z}_{1}\in \tilde{Z_{1}}\subset \tilde{T_{1}}=T_{1}\times S_{1}$. Therefore $\tilde{z}_{1}= t_{1}s_{1}$ for $t_{1}\in T_{1},\, s_{1}\in S_{1}$. Also $z_{1}= \tilde{z}^{2}_{1}= t^{2}_{1}s^{2}_{1}\in Z_{1}\subset T_{1}$. We see that $s^{2}_{1}=1$. Since $\tilde{T}_{1}$ is torsion free it follows that $s_{1}=1$ and $\tilde{z}_{1}\in T$.
\end{proof}

\subsubsection{Irreducible representation $\tilde{\pi}$ of $\tilde{G}$} In this section, we show that there exists an irreducible representation $\tilde{\pi}$ of $\tilde{G}$ which contains $\pi$ with multiplicity one on restriction to $G$.\\

The main idea behind the proof is Theorem~\ref{Gelbart-Knapp-2}. We first extend the representation $\pi$ to an irreducible representation $\pi\nu$ of $\tilde{Z}G$ and show that the group $\tilde{G}/\tilde{Z}G$ is finite abelian. Before we continue, we recall a result of Serre which we use in proving the finiteness of $\tilde{G}/\tilde{Z}G$.

\begin{proposition}[Serre]\label{Serre's theorem} If $A$ is a finite $\Gamma$ module, $H^{n}(\Gamma, A)$ is finite for every $n$.
\end{proposition}
\begin{proof} See Proposition 14, Sec. 5.1 in \cite{Ser}.
\end{proof}
%\vspace{0.5 cm}

Let $(\pi,W)$ be an irreducible representation of $G$ and $\nu$ be a self-dual character of $\tilde{Z}$ extending the central character $\omega_{\pi}$. Let $\pi\nu: \tilde{Z}G\rightarrow \GL(W)$ be defined as $(\pi\nu)(\tilde{z}g)= \nu(\tilde{z})\pi(g)$. Clearly, $\pi\nu$ is a well-defined irreducible representation of $\tilde{Z}G$. The irreducibility of $\pi\nu$ is trivial. To see $\pi\nu$ is well-defined, suppose $z_{1}g_{1}= z_{2}g_{2}$. Then
\begin{align*}
(\pi\nu)(z_{1}g_{1})&= (\pi\nu)(z_{1}z_{2}^{-1}z_{2}g_{1})\\
&= \nu(z_{1}z_{2}^{-1})\nu(z_{2})\pi(g_{1})\\
&= \omega_{\pi}(z_{1}z_{2}^{-1})\nu(z_{2})\pi(g_{1})\quad (\text{since}\, z_{1}z_{2}^{-1} \in \tilde{Z}\cap G)\\
                                                           %\textrm{ and $\nu$ is an extension of $\omega_{\pi}$})\\
&= \omega_{\pi}(z_{1}z_{2}^{-1})\nu(z_{2})\pi(z_{1}^{-1}z_{2}g_{2})\\
&= \omega_{\pi}(z_{1}z_{2}^{-1})\nu(z_{2})\omega_{\pi}(z_{1}^{-1}z_{2})\pi(g_{2})\\
&= (\pi\nu)(z_{2}g_{2}).
\end{align*}
\vspace{0.5 cm}
We now prove the finiteness of $\tilde{G}/\tilde{Z}G$.
%\vspace{0.3 cm}
\begin{theorem}
$\tilde{G}/\tilde{Z}G$ is a finite abelian group.
\end{theorem}

\begin{proof} Clearly, $\tilde{G}= \tilde{T}G$. Now
\begin{align*}
\tilde{G}/\tilde{Z}G &= \tilde{T}G/\tilde{Z}G\\
&= (\tilde{T}\tilde{Z}G)/\tilde{Z}G\\
&= \tilde{T}/(\tilde{T}\cap\tilde{Z}G)\\
&= \tilde{T}/\tilde{Z}T.
\end{align*}
It follows that $\tilde{G}/\tilde{Z}G$ is abelian. Let $\bar{F}$ be the algebraic closure of $F$ and $\Gamma =
\Gal(\bar{F}/F)$. Let $m: T(\bar{F})\times
\tilde{Z}(\bar{F})\rightarrow \tilde{T}(\bar{F})$ be the
multiplication map. This map is surjective with $\Ker(m)= \{(z,z^{-1})\mid z\in Z(\bar{F})\}$ (follows by considering the dimensions). Considering $Z(\bar{F})$ embedded diagonally in $T(\bar{F})\times \tilde{Z}(\bar{F})$, we get the following exact sequence of abelian groups
\begin{displaymath}
1 \longrightarrow Z(\bar{F}) \longrightarrow T(\bar{F})\times \tilde{Z}(\bar{F})\overset{m}\longrightarrow \tilde{T}(\bar{F})\longrightarrow 1.
\end{displaymath}
$\Gamma$ clearly acts on these groups and applying Galois cohomology, we get a long exact sequence of cohomology groups
\begin{gather*}
1 \longrightarrow Z(\bar{F})^{\Gamma} \longrightarrow T(\bar{F})^{\Gamma}\times \tilde{Z}(\bar{F})^{\Gamma}\longrightarrow \tilde{T}(\bar{F})^{\Gamma}\longrightarrow H^{1}(\Gamma, Z(\bar{F}))\\\longrightarrow H^{1}(\Gamma, T(\bar{F})\times \tilde{Z}(\bar{F}))
\longrightarrow H^{1}(\Gamma, \tilde{Z}(\bar{F}))\longrightarrow \cdots
\end{gather*}

We note that $H^{1}(\Gamma, T(\bar{F})\times \tilde{Z}(\bar{F}))=1$ (Hilbert 90) and we get the short exact sequence

\begin{equation}\label{eq:3}
1 \longrightarrow Z \longrightarrow T \times \tilde{Z}\overset{m}\longrightarrow \tilde{T}\overset{\varphi}\longrightarrow H^{1}(\Gamma, Z(\bar{F}))\longrightarrow 1.
\end{equation}
From (~\ref{eq:3}) it follows that $\varphi$ is surjective, $\Image(m)= T\tilde{Z}= \Ker(\varphi)$, and $\tilde{T}/\tilde{Z}T\simeq H^{1}(\Gamma, Z(\bar{F}))$. It is enough to show that $H^{1}(\Gamma, Z(\bar{F}))$ is finite. Let $Z^{\circ}$ be the identity component of the algebraic group $Z$. Consider the short exact sequence
\begin{equation}\label{eq:4}
1\longrightarrow Z^{\circ}(\bar{F})\longrightarrow Z(\bar{F})\longrightarrow Z(\bar{F})/Z^{\circ}(\bar{F})\longrightarrow 1.
\end{equation}
Applying Galois cohomology again to \eqref{eq:4}, we get the sequence
\begin{gather*}
1\longrightarrow Z^{\circ}\longrightarrow Z\longrightarrow Z/Z^{\circ}
\longrightarrow H^{1}(\Gamma, Z^{\circ}(\bar{F})) \longrightarrow H^{1}(\Gamma, Z(\bar{F}))\\\longrightarrow H^{1}(\Gamma, Z(\bar{F})/Z^{\circ}(\bar{F}))\longrightarrow \cdots
\end{gather*}

\noindent Since $Z^{\circ}(\bar{F})$ is connected, we have
$H^{1}(\Gamma, Z^{\circ}(\bar{F}))=1$ and it follows that $H^{1}(\Gamma, Z(\bar{F})) \hookrightarrow H^{1}(\Gamma,
Z(\bar{F})/Z^{\circ}(\bar{F}))$. Since $F$ is a local field of
$\ch \, 0$ and $Z(\bar{F})/Z^{\circ}(\bar{F})$ is a finite
abelian group, $H^{1}(\Gamma,
Z(\bar{F})/Z^{\circ}(\bar{F}))$ is finite (Proposition~\ref{Serre's theorem}). Hence the result follows.
\end{proof}

By Theorem~\ref{Gelbart-Knapp-1} and Theorem~\ref{Gelbart-Knapp-2}, we get an irreducible representation $(\tilde{\pi},V)$ of $\tilde{G}$ which breaks up as a finite direct sum of distinct irreducible representations $\pi_{1}, \ldots, \pi_{k}$ each occurring with the same multiplicity $m$ on restriction to $\tilde{Z}G$ and contains $\pi\nu$ as a constituent. Without loss of generality, we assume that $\pi_{1}\simeq \pi\nu$. To simplify notation, we again denote the restriction of $\pi_{i}$'s to $G$ by $\pi_{i}$ so that $\displaystyle \tilde{\pi}|_{G}= m\pi_{1}\oplus
m\pi_{2}\oplus \ldots \oplus m\pi_{k}$ and $\pi\simeq \pi_{1}$. We now show that each $\pi_{i}$ occurs with multiplicity one in $\tilde{\pi}|_{G}$.

\begin{lemma} The representation $(\tilde{\pi}, V)$ of $\tilde{G}$ is generic and each irreducible representation $\pi_{i}$ occurs with multiplicity one.
\end{lemma}

\begin{proof} Since $(\pi,W)$ is generic, there exists a non-degenerate character $\psi$ of $U$ such that $\Hom_{G}(\pi, \Ind_{U}^{G}\psi)\neq 0$. It is enough to show that $\Hom_{\tilde{G}}(\tilde{\pi}, \Ind_{\tilde{U}}^{\tilde{G}}\psi)\neq 0$.
Observe that $\tilde{U}= U\subset G$. Consider the restriction $\tilde{\pi}|_{U}$ of $\tilde{\pi}$.
Since $\pi$ is generic, $\Hom_{U}(\pi|_{U}, \psi)\neq 0$. It follows that
\begin{displaymath}
\Hom_{U}(\tilde{\pi}|_{U}, \psi)\simeq \Hom_{\tilde{G}}(\tilde{\pi}, \Ind_{\tilde{U}}^{\tilde{G}}\psi)\neq 0.
\end{displaymath}
Indeed,
\begin{align*}
\Hom_{U}(\tilde{\pi}|_{U}, \psi)&= \Hom_{U}((\tilde{\pi}|_{G})|_{U}, \psi)\\
&= \Hom_{U}((m\pi_{1}\oplus m\pi_{2}\oplus \cdots \oplus m\pi_{k})|_{U}, \psi)\\
&= m\bigoplus_{i=1}^k \Hom_{U}(\pi_{i}|_{U}, \psi)\\
&\neq 0.
\end{align*}
As $\tilde{\pi}$ is generic it follows that $\displaystyle dim (\Hom_{\tilde{G}}(\tilde{\pi}, \Ind_{U}^{\tilde{G}}\psi))= 1$. Thus by Frobenius Reciprocity, $m=1$.
\end{proof}

\subsubsection{Choosing $\tilde{\pi}$ with non-zero $\tilde{I}$ fixed vectors} In this section, we show that the representation $\tilde{\pi}$ can be modified in such a way that it has non-zero vectors fixed under an Iwahori subgroup $\tilde{I}$ in $\tilde{G}$.

\begin{lemma}\label{lifting from I tilde to G tilde} Suppose that $\tau_{1}$ is a linear character of $\tilde{I}$ which is trivial on $I$. Then $\tau_{1}$ extends to a linear character
$\tilde{\tau}$ of $\tilde{G}$ which is trivial on $G$.
\end{lemma}

\begin{proof} Let $I^{-}= I\cap \bar{U}$ and $I^{+}= I\cap U$. We know that $I=I^{-}T_{\circ}I^{+}$. Since $U=\tilde{U}$ we have $\tilde{I}= I^{-}\tilde{T}_{\circ}I^{+}$. Now
\begin{align*}
\tilde{I}/I&= \tilde{T}_{\circ}I/I \,\, (\, \tilde{T}_{\circ}\, \text{normalizes}\,\, I^{-},\, I^{+}\,)\\
&= \tilde{T}_{\circ}/\tilde{T}_{\circ}\cap I\\
&= \tilde{T}_{\circ}/T_{\circ}.
\end{align*}
It follows that we can consider $\tau_{1}$ as a linear character of $\tilde{T_{\circ}}$ which is trivial on $T_{\circ}$.
We first extend $\tau_{1}$ to a character $\tilde{\tau}_{1}$ of $\tilde{T}$ by making it trivial on $\tilde{T}_{1}$,
i.e., $\tilde{\tau}_{1}(\tilde{t}_{\circ}\tilde{t}_{1})= \tau_{1}(\tilde{t}_{\circ})$.
Now define an extension $\tilde{\tau}$ of $\tilde{\tau}_{1}$ to $\tilde{G}$ as
$\tilde{\tau}(\tilde{t}g)= \tilde{\tau}_{1}(\tilde{t}), \tilde{t}\in \tilde{T}, g\in G$ (this is possible since $\tilde{G}=\tilde{T}G$). Using
\begin{displaymath}
\tilde{G}/G = \tilde{T}G/T = \tilde{T}/T = \tilde{T}_{\circ}/T_{\circ} \times \tilde{T}_{1}/T_{1}
\end{displaymath}
it follows that $\tilde{\tau}$ is well-defined and a character of $\tilde{G}$.
\end{proof}
%\vspace{0.5 cm}

\begin{theorem} The representation $(\tilde{\pi}\tau^{-1},V)$ of $\tilde{G}$
has non-zero $\tilde{I}$ fixed vectors.
\end{theorem}
\begin{proof}
Let $v\neq 0\in V$ be such that $\tilde{\pi}(i)v=v,\, \forall i\in I$ ($v$ exists since $\tilde{\pi}|_{G}\supset \pi$ and $\pi$ has non-zero vectors fixed under $I$). Let $\displaystyle V_{0}= \Span_{\mathbb{C}}\{\tilde{\pi}(k)v \mid k\in \tilde{I}\}$. Clearly, $V_{0}$ is an invariant subspace for $\tilde{I}$ and thus we get a representation $(\rho, V_{0})$ of $\tilde{I}$. Suppose $\rho= \tau_{1}\oplus\cdots\oplus\tau_{k}$, where each $\tau_{i}$ is an irreducible representation of $\tilde{I}$. We know that $1_{I}\subset \rho$. Pick an irreducible component, say $\tau_{1}$, that contains $1_{I}$. By Clifford's theorem,
$I\leq \Ker(\tau_{1})$. Since $\tilde{I}/I$ is a compact abelian group, it follows that $\tau_{1}$ is a linear character of $\tilde{I}$ which is trivial on $I$.
By Lemma~\ref{lifting from I tilde to G tilde}, $\tau_{1}$ extends to a linear character $\tilde{\tau}$ of $\tilde{G}$ trivial on $G$. Consider the irreducible representation $\tilde{\pi}\tau^{-1}$.
Clearly it has an $\tilde{I}$ fixed vector. Indeed for $w$ in the space of $\tau_{1}$ and $k\in \tilde{I}$, we have
\begin{align*}
(\tilde{\pi}\tau^{-1})(k)w &= \tau^{-1}(k)\tilde{\pi}(k)w\\
&= \tau^{-1}(k)\tau_{1}(k)w\\
&= w\,\, (\, \text{since}\, \tau|_{\tilde{I}}= \tau_{1}\,).
\end{align*}
\end{proof}
%\vspace{0.5 cm}

It is easy to see that $\tilde{\pi}\tau^{-1}|_{G}$ contains the representation $\pi$ with multiplicity one, in addition to having non-zero $\tilde{I}$ fixed vectors. To simplify notation, we will denote the representation $\tilde{\pi}\tau^{-1}$ as $\tilde{\pi}$.

\subsubsection{Sign of $\tilde{\pi}$} In this section, we attach a sign $\varepsilon{(\tilde{\pi})}$ to the representation $\tilde{\pi}$. We also give a formula to compute $\varepsilon{(\tilde{\pi})}$ and show that $\varepsilon{(\tilde{\pi})}= \varepsilon{(\pi)}$. Finally we show that $\varepsilon{(\tilde{\pi})}=1$ to complete the proof of the main theorem.\\

Consider the representation $\tilde{\pi}^{\vee}$. This is again an
irreducible representation of $\tilde{G}$ which on restriction to
$\tilde{Z}G$ and contains the representation $\pi\nu$ with
multiplicity 1 (since $\nu= \nu^{-1}$ and $\pi\simeq \pi^{\vee}$). By Theorem~\ref{Gelbart-Knapp-2}, there's a linear character $\chi$ of $\tilde{G}$ trivial on $\tilde{Z}G$ such that $\tilde{\pi}^{\vee}\simeq \tilde{\pi}\otimes \chi$. We use this isomorphism to define a non-degenerate bilinear form $[\,,\,]$ on $V$.

\begin{lemma} There exists a non-degenerate form $[\,,\,]: V \times V \rightarrow \mathbb{C}$ satisfying $[\tilde{\pi}(g)v_{1}, \tilde{\pi}(g)v_{2}]
= \chi^{-1}(g)[v_{1}, v_{2}]$.
\end{lemma}
\begin{proof} Since $\tilde{\pi}^{\vee}\simeq \tilde{\pi}\otimes \chi$, there exists a non-zero map $\phi: V \rightarrow V^{\vee}$ such that
$\tilde{\pi}^{\vee}(g)(\phi(v))= \phi((\tilde{\pi}\otimes \chi)(g)v)$. Let $\langle\,,\, \rangle : V \times V^{\vee}\rightarrow \mathbb{C}$
be the canonical $\tilde{G}$ invariant pairing. We define $[\,,\,]: V \times V \rightarrow \mathbb{C}$ as $[v_{1}, v_{2}]= \langle v_{1}, \phi(v_{2})\rangle$.
Clearly this form is non-degenerate and satisfies $[\tilde{\pi}(g)v_{1}, \tilde{\pi}(g)v_{2}]
= \chi^{-1}(g)[v_{1}, v_{2}]$. Indeed,
\begin{align*}
[\tilde{\pi}(g)v_{1}, \tilde{\pi}(g)v_{2}]&= \langle \tilde{\pi}(g)v_{1}, \tilde{\pi}(g)v_{2}\rangle\\
&= \langle \tilde{\pi}(g)v_{1}, \chi^{-1}(g)\tilde{\pi}^{\vee}(g)(\phi(v_{2}))\rangle\\
&= \chi^{-1}(g)\langle \tilde{\pi}(g)v_{1}, \tilde{\pi}^{\vee}(g)(\phi(v_{2}))\rangle\\
&= \chi^{-1}(g)\langle v_{1}, \phi(v_{2})\rangle\\
&= \chi^{-1}(g)[v_{1}, v_{2}].
\end{align*}
\end{proof}

The form $[\,,\,]$ is unique up to scalars and is easily seen to be symmetric or skew-symmetric as before, i.e.,
\begin{displaymath}
[v_{1},v_{2}]=\varepsilon{(\tilde{\pi})}[v_{2},v_{1}].
\end{displaymath}
where $\varepsilon{(\tilde{\pi})}\in \{\pm 1\}$. We call $\varepsilon{(\tilde{\pi})}$ the sign of $\tilde{\pi}$. \\

Let $[\,,\,]: V\times V\longrightarrow \mathbb{C}$ be the non-degenerate
bilinear form on $V$(obtained above). Suppose that
$[\,,\,]\big |_{W_{1}\times W_{j}} = 0, \forall j= 2, 3, \cdots,
k$, then it is easy to see that $[\,,\,]\big|_{W_{1}\times W_{1}}$ is non-degenerate. We now show that $[\,,\,]\big|_{W_{1}\times W_{j}}=0$ for $j=2, 3,
\cdots, k$.

\begin{lemma}
$[\,,\,]\big|_{W_{1}\times W_{j}}=0$,\, $\forall j=2, 3, \cdots, k$.
\end{lemma}

\begin{proof}
Suppose $[\,,\,]\big|_{W_{1}\times W_{j}}\neq 0$. Let $v\in W_{1}$ and $u\in W_{j}$ be such that $[v,u]\neq 0$. Let
$\phi(w)=\phi_{w}$ be defined as $\phi_{w}(v)= [v,w]$. Then clearly $\phi$
is a non-zero intertwining map between $\pi^{\vee}_{1}$ and $\pi_{j}$. Indeed,
\begin{align*}
(\pi^{\vee}_{1}(g)\circ \phi)(w_{j})(w_{1})&= \pi^{\vee}_{1}(g)(\phi(w_{j})(w_{1})\\
&= \phi(w_{j})(\pi_{1}(g^{-1})w_{1})\\
&= [\pi_{1}(g^{-1})w_{1}, w_{j}]\\
&= [w_{1}, \pi_{j}(g)w_{j}]\\
&= \phi(\pi_{j}(g)w_{j})(v_{1}).
\end{align*}
Since $\pi^{\vee}_{1}\simeq \pi_{1}$ and the representations $\pi_{i}$ are distinct (up to isomorphism), the lemma follows.
\end{proof}
%\vspace{0.5 cm}
We know that $[\tilde{\pi}(g)v_{1},
\tilde{\pi}(g)v_{2}]=\chi^{-1}(g)[v_{1},v_{2}],\forall g\in \tilde{G}, v_{1},
v_{2}\in V$. Now if $g\in G$ then $\chi(g)=1$ and we have
$[\tilde{\pi}(g)v_{1},\tilde{\pi}(g)v_{2}]=[v_{1},v_{2}]$. In particular if
$v_{1}\in W_{1}$ and $v_{j}\in W_{j}$, then
$[\tilde{\pi}(g)v_{1},\tilde{\pi}(g)v_{j}]=[v_{1},v_{j}]$. Since $V= W_{1}\oplus
W_{2}\oplus \cdots \oplus W_{k}, \tilde{\pi}(g)v_{1}=\pi_{1}(g)v_{1},
\tilde{\pi}(g)v_{j}=\pi_{j}(g)v_{j}$.

\begin{lemma}\label{relating signs} With notation as above, $\varepsilon(\tilde{\pi})=\varepsilon(\pi)$.
\end{lemma}

\begin{proof} Since $\chi(g)=1$ for $g\in G$, we have
\begin{align*}
[\tilde{\pi}(g)w_{1}, \tilde{\pi}(g)\dot{w_{1}}] &= [\pi_{1}(g)w_{1}, \pi_{1}(g)\dot{w_{1}}]\\
&= [w_{1},\dot{w_{1}}].
\end{align*}
Since $\pi_{1}\simeq \pi$, we see that $[\,,\,]\big
|_{W_{1}\times W_{1}}$ is $G$-invariant. Therefore
$[w_{1},\dot{w_{1}}]= \varepsilon(\pi)(w_{1},\dot{w_{1}})$. But we
also know that $[w_{1}, \dot{w_{1}}] = \varepsilon(\tilde{\pi})[\dot{w_{1}},
w_{1}]$.
\end{proof}%complete
%\vspace{0.5 cm}
By Lemma~\ref{relating signs}, it follows that the sign $\varepsilon({\pi})$ is completely determined by the sign $\varepsilon({\tilde{\pi}})$.
Since $\tilde{Z}$ is connected, applying Theorem~\ref{existence of the elt s} we get an element $s\in \tilde{T}_\circ$ such that $\alpha(s)=-1$ for all simple roots $\alpha$ of $\tilde{T}$. We will show that the sign $\varepsilon({\tilde{\pi}})$ is controlled by the central character $\omega_{\tilde{\pi}}$ and the character $\chi$. Before we proceed, we prove a lemma we need.

\begin{lemma}\label{pairing} Let $W_{1}, W_{2}$ be irreducible $K$-invariant subspaces of $V$. Let $\rho_{1}=\tilde{\pi}|_{W_{1}}$ and $\rho_{2}=\tilde{\pi}|_{W_{2}}$. Let $b: W_{2}\rightarrow W_{1}^{\vee}$ be the map $w_{2}\mapsto [-,w_{2}]$. If $b\neq 0$, then $\rho_{2}\simeq \rho_{1}^{\vee}\chi$.
\end{lemma}
\begin{proof} We first show that $b$ defines an intertwining map between $\rho_{2}$ and $\rho_{1}^{\vee}\chi$. Indeed, for $w_{1}\in W_{1}, w_{2}\in W_{2}$ and $k\in K$, we have
\begin{align*}
b(\rho_{2}(k)(w_{2}))(w_{1}) &= [w_{1}, \rho_{2}(k)w_{2}]\\
&= \chi(k)[\rho_{1}(k^{-1})(w_{1}), w_{2}]\\
&= \chi(k)b(w_{2})(\rho_{1}(k^{-1})w_{1})\\
&= \chi(k)\rho_{1}^{\vee}(k)(b(w_{2}))(w_{1}).
\end{align*}
Since $b\neq 0$, Schur's Lemma applies and the result follows.
\end{proof}
%\vspace{0.5 cm}
Let $V_{0}$ be the space of $\psi_{\tilde{K}}$ and $v_{0}\in V_{0}$. Since $[\,,\,]$ is non-degenerate, $b(v_{0})(v_{1})=[v_{1},v_{0}]\neq 0$ for some $v_{1}\in V_{1}$ where $V_{1}$ is an irreducible $\tilde{K}$-invariant subspace of $V$. We denote $\rho$ for the restriction of $\tilde{\pi}|_{V_{1}}$. By Lemma~\ref{pairing}, it follows that $\rho\simeq \psi_{\tilde{K}}^{-1}\chi$. Since $\chi$ is smooth, we can in fact choose $\tilde{K}$ such that $\chi$ is trivial on $\tilde{K}$. It follows that any vector $v_{0}\in V_{0}$ has to pair non-trivially with some vector in the space of $\psi_{\tilde{K}}^{-1}$, i.e., $[\tilde{\pi}(s)v_{0},v_{0}]\neq 0$. We will use this in the following theorem.

\begin{theorem} Let $(\tilde{\pi},V)$ be the irreducible representation of $\tilde{G}$ obtained above.
Then $\varepsilon({\tilde{\pi}})= \omega_{\tilde{\pi}}(s^{2})\chi(s)$.
\end{theorem}
\begin{proof} Clearly $s^{2}\in \tilde{Z}$. Since $\tilde{\pi}$ is generic %($\Hom_{\tilde{G}}(\tilde{\pi}, \ind_{U}^{\tilde{G}}\psi)\neq 0$),
it follows by Theorem~\ref{Rodier} that there exists a compact open subgroup $\tilde{K}$ and a character $\psi_{\tilde{K}}$ of $\tilde{K}$ such that $\psi_{\tilde{K}}$ occurs with multiplicity one in $\tilde{\pi}|_{\tilde{K}}$. Let $V_{0}$ be the space of $\psi_{\tilde{K}}$ and $0\neq v_{0}\in V_{0}$. Now
\begin{align*}
[\tilde{\pi}(s)v_{0}, \tilde{\pi}(s^{2})v_{0}] &= \omega_{\tilde{\pi}}(s^{2})[\tilde{\pi}(s)v_{0}, v_{0}] \quad (\text{since $s^{2}\in \tilde{Z}$})\\
&= \chi^{-1}(s)[v_{0}, \tilde{\pi}(s)v_{0}] \quad (\text{invariance of the form})
\end{align*}
It follows that $\varepsilon({\tilde{\pi}})= \omega_{\tilde{\pi}}(s^{2})\chi(s)$.
\end{proof}
%\vspace{0.5 cm}

Using $\tilde{\pi}$ has Iwahori fixed vectors and $s^{2}\in \tilde{T}_{\circ}$, we see that $\omega_{\tilde{\pi}}(s^{2})=1$. It will
follow that $\varepsilon({\pi})=1$ once we show that $\chi(s)=1$. We do this by showing that $\chi$ is an unramified character. Before we continue, we recall a result about intertwining maps which we need in the proof.\\

For $K$ a compact open subgroup of $\tilde{G}$ and $\rho$ an irreducible representation of $K$, we let $\hat{K}$ denote the set of equivalence
classes of irreducible smooth representations of $K$, $K^{g}= g^{-1}Kg,\, g\in \tilde{G}$ and $\rho^{g}$ the irreducible representation of $K^{g}$
defined as $x\rightarrow \rho(gxg^{-1})$.

\begin{proposition}\label{intertwiners} For $i=1,2$, let $K_{i}$ be a compact open subgroup of $\tilde{G}$ and let $\rho_{i}\in \hat{K_{i}}$.
Let $(\Pi,V)$ be an irreducible representation of $\tilde{G}$ which contains both $\rho_{1}$ and $\rho_{2}$.
There then exists $g\in \tilde{G}$ such that $\displaystyle \Hom_{K_{1}^{g}\cap K_{2}}(\rho_{1}^{g}, \rho_{2})\neq 0$.
\end{proposition}
\begin{proof}
We refer the reader to ~\cite{BusHen} (Chapter 3, Section 11, Proposition 1) for a proof of the above proposition.
In \cite{BusHen}, the authors prove the result for $\GL(2,F)$. The same proof works in the case of any connected reductive group.
\end{proof}
%\vspace{0.5 cm}

\begin{theorem} The character $\chi$ is an unramified character. In particular $\chi(s)=1$.
\end{theorem}

\begin{proof}
We know that $\tilde{\pi}^{\vee}\simeq \tilde{\pi}\otimes\chi$. Since $\tilde{\pi}$ has
non-trivial $\tilde{I}$ fixed vectors it follows that $\tilde{\pi}^{\vee}$
and hence $\tilde{\pi}\otimes\chi$ has non-trivial $\tilde{I}$ fixed vectors.
Therefore $(\tilde{\pi}\otimes\chi)\big |_{\tilde{I}}\supset 1$ and $(\tilde{\pi}\otimes\chi) \big
|_{\tilde{I}}\supset \chi$. By Proposition~\ref{intertwiners} there exists $g\in
\tilde{G}$ such that $\Hom_{\tilde{I}^{g}\cap \tilde{I}}(1^{g},
\chi)\neq 0$. Since $\displaystyle \tilde{G}= \coprod_{w\in
\widetilde{W}}\tilde{I}w\tilde{I}$ we see that $\displaystyle
\Hom_{\tilde{I}^{w}\cap \tilde{I}}(1^{w}, \chi)\neq 0$ when $g\in
\tilde{I}w\tilde{I}$. From this it follows that $\chi(h)=1,
\forall h\in \tilde{I}^{w}\cap \tilde{I}$. Since $\tilde{T}_{\circ}\subset
\tilde{I}^{w}\cap \tilde{I}$ it follows that $\chi$ is unramified.
\end{proof}

\section*{Acknowledgements}
I would like to thank my advisor Alan Roche for his constant help and encouragement throughout this project. I would also like to thank him for his comments and suggestions in improving the presentation of this article.

\bibliographystyle{amsplain}
\bibliography{sdrpaper}
\end{document}